\documentclass[11pt]{article}
\usepackage[latin1]{inputenc}
\usepackage{epsfig}
\usepackage{amsmath,amssymb}
\usepackage[active]{srcltx}
\usepackage{dsfont}
\usepackage[titles]{tocloft}
\usepackage{amsthm}
\usepackage{bbm}
\usepackage{tikz}
\usetikzlibrary{shapes}
\usepackage[colorlinks=true,linkcolor=blue,urlcolor=blue,citecolor=red, hyperfigures=false]{hyperref}
\let\inf\relax \DeclareMathOperator*\inf{\vphantom{p}inf}
\let\max\relax \DeclareMathOperator*\max{\vphantom{p}max}
\let\min\relax \DeclareMathOperator*\min{\vphantom{p}min}

\numberwithin{equation}{section}

\newtheorem{theoreme}{Theorem}[section]
\newtheorem{proposition}[theoreme]{Proposition}

\newtheorem{corollaire}[theoreme]{Corollary}
\newtheorem{remarque}[theoreme]{Remark}

\newtheorem{lemme}[theoreme]{Lemma}
\newtheorem{definition}[theoreme]{Definition}

\newcommand{\RR}{\ensuremath{\mathbb R}}

\newcommand{\PP}{\ensuremath{\mathbb P}}

\newcommand{\EE}{\ensuremath{\mathbb E}}
\newcommand{\NN}{\ensuremath{\mathbb N}}

\newcommand{\D}{\ensuremath{\mathcal D}}

\newcommand{\ind}{\ensuremath{\mathds 1}}
\newcommand{\G}{\ensuremath{\mathcal G}}

\newcommand{\xx}{\ensuremath{\mathbf{x}}}

\newcommand{\De}{\Delta}
\newcommand{\la}{\lambda}

\newcommand{\de}{\delta}

\newcommand{\ep}{\varepsilon}

\newcommand{\si}{\sigma}
\newcommand{\ga}{\gamma}

\newcommand{\sh}{\ensuremath{\hat{s}}}

\newcommand{\two}{I\!\! I}
\title{Optimal strategies in zero-sum repeated games with incomplete information: the dependent case}

\author{\scshape\MakeUppercase 
Fabien Gensbittel\footnote{Toulouse School of Economics, University of Toulouse Capitole, Toulouse, France, email: \href{mailto:fabien.gensbittel@tse-fr.eu}{fabien.gensbittel@tse-fr.eu}}, 
\scshape\MakeUppercase Miquel Oliu-Barton\footnote{Universit\'e Paris-Dauphine, PSL Research University, CNRS, CEREMADE, 75016 Paris, France, email: \href{mailto:miquel.oliu.barton@normalesup.org}{miquel.oliu.barton@normalesup.org}}\\[1cm]
}
\date{\today}

\begin{document}
\maketitle 

\paragraph{Abstract.} 
Using the duality techniques introduced by De Meyer (1996a, 1996b), Rosenberg (1998) and De Meyer and Marino (2005) provided an explicit construction for optimal strategies in repeated games with incomplete information on both sides, in the independent case. In this note, we extend both the duality techniques and the construction of optimal strategies to the dependent case.

\setcounter{tocdepth}{2}
\setlength{\cftbeforesubsecskip}{-2ex}
\setlength{\cftbeforesecskip}{-2ex}

\section{Introduction}

We consider here a zero-sum repeated game with incomplete information on both sides, in the spirit of Aumann and Maschler \cite{AM95}. Let $K$ (resp. $L$) be the finite set of types of Player $1$ (resp. $2$), and let $\pi$ be a probability distribution over $K\times L$. To any pair $(k,\ell)$ corresponds a matrix game $G^{k\ell}:I\times J\to \RR$, where $I$ (resp. $J$) is the finite set of actions of Player $1$ (resp. $2$). The game is played as follows. First, a pair $(k,\ell)\in K\times L$ is drawn with the probability distribution $\pi$. 
Player $1$ (resp. $2$) is informed only of $k$ (resp. $\ell$). Then, the game $G^{k\ell}$ is played repeatedly. At each stage $m\geq 1$, the players choose actions $(i_m,j_m)\in I\times J$, which produces a stage-payoff $G^{k\ell}(i_m,j_m)$. Actions are publicly observed after each stage.
For any initial distribution $\pi$ and any sequence of non-negative weights $\theta=(\theta_m)_m$, we consider the game $\G_\theta(\pi)$ in which the overall payoff is the expected $\theta$-weighted sum of the stage-payoffs $\sum_{m\geq 1} \theta_m G^{k\ell}(i_m,j_m)$ and where $\pi$ stands for the probability distribution of $(k,\ell)$. This game has a value, denoted by $v_\theta(\pi)$. The particular case where, for some $n\in\NN^*$, one has $\theta_m=\frac{1}{n}\ind_{\{m\leq n\}}$ for all $m\geq 1$ corresponds to the classical $n$-stage repeated games. Similarly, the case where $\theta_m=\la (1-\la)^{m-1}$ for all $m\geq 1$ and some $\la \in(0,1]$ corresponds to $\la$-discounted repeated games. We use then the notation $\G_n(\pi)$ and $v_n(\pi)$, and $\G_\la(\pi)$ and $v_\la(\pi)$,  respectively.

This model was analyzed by Mertens and Zamir in \cite{MZ71}. Their main result was the existence of $v_n(\pi)$ and $v_\la(\pi)$ and their convergence (respectively, as $n$ goes to $+\infty$ and as $\la$ vanishes) to the unique solution of a system of functional equations. The proof of this result was based on the introduction of the specific notion of $I$-concavity 
for the value function $\pi\mapsto v_n(\pi)$, which can be described as follows. Any probability $\pi$ over the product set $K\times L$ can be decomposed as a pair $(p,Q)$ where $p$ is the marginal probability on $K$ and $Q$ is a matrix of conditional probabilities on $L$ given $k\in K$. This decomposition can be expressed 
 as $\pi=p\otimes Q$ where $\otimes$ denotes the direct product. 
One may then consider $v_{n}$ as a function of $(p,Q)$, and show that $p\mapsto v_n(p,Q):=v_n(p\otimes Q)$ is a concave function for any fixed $Q$.
 A dual notion of $\two$-convexity was also introduced and the notions of  $I$-concave and $\two$-convex envelopes were the building blocks of the system of functional equations characterizing the limit value. Based on this characterization, a construction of asymptotically optimal strategies (i.e. strategies being almost optimal in $\G_n(\pi)$, with an error term vanishing as the number of stages tends to $+\infty$) 
 was obtained by Heuer \cite{heuer92}. The convergence of the values $v_\theta(\pi)$ for a general evaluation, as $\max_{m\geq 1}\theta_m$ tends to $0$, and the construction of asymptotically optimal strategies in this case were obtained by Oliu-Barton \cite{OB_asym,OB17}.

In addition to their main result, Mertens and Zamir  \cite{MZ71} also established a recursive formula for $v_{n}(\pi)$ and $v_\la(\pi)$ in terms of the conditional probabilities on $K\times L$ induced by the players' strategies at the first stage, and the extension to a general evaluation $\theta$ is straightforward. 
Though very useful for studying the values, the formula cannot be used by the players 
for the simple reason that none of them can actually compute these conditional probabilities. There is, however, one important exception: games with incomplete information on one side. Indeed, when Player 2 has no private information, Player 1 controls and observes the conditional probabilities while Player 2 does not. As a consequence, the former, and not the latter, can 
use the recursive formula satisfied by the values to construct an optimal strategy. This game is denoted by $\G_\theta(p)$ where $p$ is the probability distribution of $k$. 
The dual game was introduced by De Meyer in \cite{dm96,dm96b} in order to construct an optimal strategy for Player 2. 
The idea of the dual game is to consider a game with vector payoffs: for each realised pair of actions $(i,j)$, the uninformed player knows the vector $(G^k(i,j))_{k\in K}$. Like in approachability theory, an optimal strategy for Player 2 is one that ensures that the $\theta$-weighted sum of payoffs lies in an appropriate subset of $\RR^K$. This set depends on a well-chosen dual variable $x\in \RR^K$ which replaces the unknown type of Player $1$ in the following sense: Player $2$ can choose his opponent's type to be $k$ at a cost $x^k$. 
De Meyer \cite{dm96} proves that the values of the dual game $w_\theta(x)$ satisfy a recursive formula in terms of the dual variable, and that Player $2$ can use this formula to construct an optimal strategy in the dual game. More importantly, this strategy is an optimal strategy in $\G_\theta(p)$ provided that $x$ belongs to the sub-differential of the concave function $p'\mapsto v_\theta(p')$ at $p$.

The duality techniques were extended by Rosenberg \cite{rosenberg98}, Sorin \cite{sorin02} and De Meyer and Marino \cite{DM05} for repeated games with incomplete information on both sides, in the special case of independent initial probabilities, i.e. $\pi=p\otimes q$, for some probabilities $p$ over $K$  and $q$ over $L$.
As both players are uninformed about their opponent's type, one needs to consider two dual games, one for each player. The first dual game is related to the Fenchel conjugate of the function $p' \rightarrow v_{\theta}(p',q):=v_\theta(p\otimes q)$, where $q$ is a fixed probability over $L$. Rosenberg \cite{rosenberg98} proved that its value $w_\theta(x,q)$ satisfies a recursive formula in terms of the dual variable $x$ and the conditional probabilities over $L$ induced by the strategy of Player 2. As these two variables are accessible to Player $2$, the latter can use this formula to construct an optimal strategy in the dual game, and this strategy is an optimal strategy in the game  
$\G_\theta(\pi)$ provided that $x$ belongs to the sub-differential of the concave function $p'\mapsto v_\theta(p' \otimes q)$ at $p$, where $p$ and $q$ are such that $\pi=p\otimes q$. An alternative formula with similar properties was obtained more recently by De Meyer and Marino \cite{DM05}, who also considered the case of infinite action spaces.
The second dual game is constructed in a symmetric manner, and provides an optimal strategy of Player $1$.  It is worth mentioning that, unlike the asymptotic results from Heuer \cite{heuer92} and Oliu-Barton \cite{OB_asym}, the constructions of Rosenberg \cite{rosenberg98} and De Meyer and Marino \cite{DM05} 
provide optimal strategies for repeated games with a fixed evaluation (namely, $n$-stage and $\la$-discounted games).

In the present paper, we extend the results from Rosenberg \cite{rosenberg98} to the so-called dependent case. That is, we provide a recursive formula for the values of the dual games, from which we deduce the computation of explicit optimal strategies for the players in the repeated game with incomplete information $\G_\theta(\pi)$, for any evaluation $\theta$ and any probability $\pi$ on $K\times L$. Our construction can be extended, word for word, to stochastic games with incomplete information, as long as the incomplete information concerns the payoff function, but not the transition function. 
 Extending the duality techniques to the dependent case was never done before; albeit not technically difficult, the extension requires some new ideas, such as considering an intermediate step: first, the type of one player is drawn according to the corresponding marginal law, and then the other type is drawn according to the conditional law. These considerations are crucial in the proof of our main result in order to prove the convexity of some auxiliary functions (see Remark \ref{rem_proof}). Besides, contrary to \cite{rosenberg98}, who considered $n$-stage games and $\la$-discounted games separately, and games with incomplete information on one and two sides separately too, we present here a unified approach.
Let us also point out that the approach proposed by De Meyer and Marino \cite{DM05}, which was designed to handle games with infinite action spaces, does not seem well-suited to analyze the dependent case. Indeed, when applied to this case, their method requires the introduction of an additional dual variable, and thus results in a substantially more complicated dual recursive formula (see \cite[Chapter 4]{OB13}).

\subsection{Main results} 
As both players have symmetric roles, we will only state our results on one side. Namely, we will focus on the optimal strategies of Player 2. 

In order to state our main results, we need the following notation:
\begin{itemize}
\item For a non-empty finite set $X$, $\Delta(X)$ denotes the set of probabilities over $X$, and is identified with the canonical simplex in $\RR^X$. 
\item $\NN^*$ denotes the set of positive integers, and $\De(\NN^*)$ is the set of nonnegative sequences $(a_m)_{m\in \NN^*}$ satisfying $\sum_{m\geq 1}a_m=1$. 
\item For any $\theta\in \De(\NN^*)$ satisfying $\theta_1<1$ we denote by $\theta^+\in \De(\NN^*)$ the sequence $(\frac{\theta_m}{1-\theta_1})_{m\geq 1}$. 
\item For any $p\in \De(K)$ and $Q\in \De(L)^K$ we denote by $p\otimes Q$ the probability on $K\times L$ induced by $p$ and $Q$, i.e. $(p\otimes Q)(k,\ell)=p^kQ(\ell\,|\,k)$ for all $(k,\ell)$. For any $\tau \in \Delta(J)^L$, we denote by $\PP^{pQ}_{\tau}$ the probability over $K\times L\times J$ induced by $(p,Q,\tau)$. For every $j \in J$, $\PP^{pQ}_{\tau}(j)$ denotes the marginal probability of $j$ and $Q_j$ the matrix of conditional probabilities over $L$ given $k \in K$, conditionally on $j$. 
\item For any $(k,i,Q,\tau) \in K\times I  \times \Delta(L)^K\times \Delta(J)^L$, we define 
\[ G^{kQ}_{i\tau}= \sum_{j \in J} \sum_{\ell \in L} G^{k\ell}_{ij}\tau^\ell(j)Q(\ell|k) .\]

\end{itemize}

For any evaluation $\theta$, any dual variable $x$ and any matrix of conditional probabilities $Q$, 
we denote by $w_\theta(x,Q)$ the value of the dual game corresponding to the game $\G_\theta(p,Q)$ from the perspective of Player 2. We can now state our main result.

\begin{theoreme}[Dual recursive formula]\label{thmmain} For all $(x,Q)\in \RR^K \times \De(L)^K$ and $\theta\in \De(\NN^*)$ one has:
\begin{eqnarray*}\label{eee2}
w_{\theta}(x,Q)&=&\min_{\tau\in \De(J)^L} \min_{(x_{ij})_{ij}\in (\RR^K)^{I\times J}} \max_{(k,i)\in K \times I} \\
&& \left\{ 
\theta_1 G^{kQ}_{i \tau}+(1-\theta_1)\sum\nolimits_{j\in J}\PP^{kQ}_{\tau}(j) \left( w_{\theta^+}(x_{ij},Q_j)
+ x_{ij}^k \right) -x^k \right\}.
\end{eqnarray*}
\end{theoreme}

\begin{corollaire}\label{cormain}Player $2$ can construct an optimal strategy in $\G_\theta(\pi)$ by using the \emph{dual recursive formula}, starting from an appropriate pair $(x,Q)$, namely $Q$ is the matrix of conditionals corresponding to $\pi$ and $x$ belongs to the sub-differential of $p'\mapsto v_\theta(p'\otimes Q)$ at $p$ such that $\pi=p\otimes Q$.
\end{corollaire}

\subsection{Outline of the paper} Section 2 is devoted to introduce the duality techniques in all its generality. In particular, we show how to deal with the dependent case. In Section 3 we introduce repeated games with incomplete information on both sides. Section 4 is devoted to prove our main results. 
In Section 5 we provide some comments on the extensions of our results to two classes of dyamics games, stochastic games and differential games.

\section{Duality techniques}\label{applications}
For any pair of sets $S$ and $T$ and any function $g:S\times T\to \RR$, we denote by $(S,T,g)$ the zero-sum game where $S$ is the set of strategies of Player $1$, $T$ is the set of strategies of Player $2$ and $g$ is the payoff.
The maxmin and minmax of $(S,T,g)$ are given by:
$$v^-:=\sup_{s\in S}\inf_{t\in T}g(s,t)\quad\text{ and }\quad v^+:=\inf_{t\in T}\sup_{s\in S}g(s,t),$$ 
and the game is said to have a value if $v^-=v^+$. 
 An $\ep$-optimal strategy for Player 1 is an element $s_\ep\in S$ satisfying 
$\inf_{t\in T}g(s_\ep,t)\geq v^- -\ep$. Similarly, $t_\ep\in T$ is an $\ep$-optimal strategy for Player $2$ if  $\sup_{s\in S}g(s,t_\ep)\leq v^+ +\ep$. Note that $\ep$-optimal strategies exist for all $\ep>0$ but not necessarily for $\ep=0$.

The aim of this section is to recall some properties of the \emph{dual game}, introduced by De Meyer in \cite{dm96,dm96b} to study repeated games with incomplete information on one side. We follow the presentation given in \cite[Chapter 2]{sorin02}. Throughout this section, $S$ and $T$ denote two convex sets, $K$ and $L$ are two finite sets and $G^{k\ell}:S\times T\to \RR$ is a payoff function for each $(k,\ell)\in K\times L$ that is bi-linear and bounded, i.e. $\sup_{s,t}|G^{k\ell}(s,t)|<+\infty$.

\subsection{Incomplete information on one side} \label{oneside}
Let us start by considering the case $L=\{\ell\}$, and set $G^{k}:=G^{k\ell}$ for all $k\in K$ to simplify the notation. 
To the collection of zero-sum games (henceforth, games) $\{(S,T,G^k),\ k\in K\}$, we associate two families of games, the so-called primal and dual games, parameterised in terms of a probability $p\in \De(K)$ and a vector $x\in \RR^K$, respectively. 
\paragraph{The primal game $\G(p)$.} 
To every probability distribution $p\in \De(K)$ corresponds a game with  incomplete information on one side, defined as follows:
\begin{itemize}
 \item Before the play, $k\in K$ is chosen according to $p$ and told to Player $1$. 
\item Then, the game $(S,T,G^{k})$ is played, i.e. Player 1 chooses $s\in S$, Player 2 chooses $t\in T$ (both choices being independent and simultaneous), 
and the payoff is $G^k(s,t)$. 
\end{itemize}

The set of strategies of Player $1$ is $S^K$, the set of strategies of Player $2$ is $T$, and the payoff function is given by:
$$\ga(p,\hat{s},t):=\sum_{k \in K}p^k G^k(\hat{s}^k,t).$$ 
The game $(S^K,T,\ga(p,\,\cdot))$ is denoted by $\G(p)$ and will be referred as the \emph{primal game}.
The {maxmin} and {minmax} 
of $\G(p)$ are given, respectively, by: 
\begin{eqnarray*}
v^-(p)& :=& \sup_{\hat{s}\in S^K}\inf_{t\in T}\ga(p,\hat{s},t),\\
v^+(p)&:=&\inf_{t\in T}\sup_{\hat{s}\in S^K}\ga(p,\hat{s},t).
\end{eqnarray*}
\emph{Concavity and continuity}: 
The maps $p\mapsto v^\pm(p)$ are concave and 
continuous on $\De(K)$.

\paragraph{The dual game $\D[\G](x)$.} 
To every vector $x\in \RR^K$ corresponds the dual game $\D[\G](x)$, 
a modified version of the primal game $\G(p)$ where Player 1 can \emph{choose} the parameter $k\in K$ at a cost $x^k$. Formally, the set of strategies of Player 1 in the dual game is $\De(K)\times S^K$, the set of strategies of Player 2 is $T$, and the payoff function is given by: 
$$h[x](p,\hat s,t):=\sum_{k\in K} p^k G^k(\hat{s}^k,t)-\langle p,x\rangle.$$ 
Let $w^-(x)$ and $w^+(x)$ denote, respectively, the {maxmin} and the {minmax}
of $\D[\G](x)$, i.e.
\begin{eqnarray*}
w^-(x)& :=& \sup_{(p,\hat{s})\in \De(K)\times S^K}\, \inf_{t\in T}\ h[x](p,\hat s,t),\\
w^+(x)&:=&\inf_{t\in T}\, \sup_{(p,\hat{s})\in \De(K)\times S^K}\ h[x](p,\hat{s},t).
\end{eqnarray*}
\emph{Convexity and continuity}: The maps $x\mapsto w^\pm(x)$ are convex and continuous on $\RR^K$.

The values of the primal game and the values of the dual game are essentially Fenchel conjugates from each other. To be more precise about the link between the functions $v^\pm$ and $w^\pm$, one needs to 
introduce two closely related convex transforms: the lower and the upper conjugates. 
\begin{definition} Let $f:\RR^K\to \overline\RR:=\RR\cup\{-\infty,+\infty\}$. 
Define its upper and lower conjugates $f^\sharp, f^\flat:\RR^K\to \overline\RR$ by:
\begin{eqnarray*}f^\sharp(x)&:=&
\sup_{y\in \RR^K}f(y)-\langle y, x\rangle,\\
f^\flat(y)&:=&
\inf_{x\in \RR^K} f(x)+\langle x,y \rangle.
\end{eqnarray*}
\end{definition}
\noindent Recall that the Fenchel conjugate of $f$ is given by $f^*(x)=\sup_{y\in \RR^K}\langle y, x\rangle-f(y)$ so that: 
\[f^\sharp(x)=(-f)^*(-x), \quad \text{ and }\quad f^\flat(y)=-f^*(-y),\quad \forall x,y\in \RR^K.\]
\noindent Recall also the usual definition of the superdifferential of $f$ at $x$: 
\[\partial^+ f(x):= \{y \in \RR^K \, | \, f(x)+\langle y, x'-x\rangle \geq f(x'), \ \forall x'\in \RR^K\}.\] 
\begin{theoreme}\label{th1} For all $(x,p)\in \RR^K\times \De(K)$ one has:
$$w^{\pm}(x)=(v^\pm)^\sharp(x)\quad \text{ and } \quad 
 v^{\pm}(p)=(w^\pm)^\flat(p)$$
 where the functions $v^\pm$ are extended to $\RR^K$ by $-\infty$ in $\RR^K\backslash \De(K)$.
 \end{theoreme} 
\begin{corollaire}\label{th2} Let $(x,p)\in \RR^K\times \De(K)$ be such that $x\in \partial^+ v^+(p)$. Then, any $\ep$-optimal strategy $t_\ep\in T$ of Player 2 in $\D[\G](x)$ is also an $\ep$-optimal strategy of Player 2 in $\G(p)$. 
\end{corollaire}

\subsection{Incomplete information on both sides} \label{dualide}
Consider now the general case where $K$ and $L$ may contain more than one element. 
To the collection of games $\{(S,T,G^{k\ell}),\ (k,\ell)\in K\times L\}$ one can associate a family of games with incomplete information on both sides, as before. 

\paragraph{The primal game $\G(\pi)$.} For any $\pi\in \De(K\times L)$, consider the following primal game, denoted by $\G(\pi)$: 
\begin{itemize}
 \item Before the play, a couple $(k,\ell)\in K\times L$ is chosen according to $\pi$, Player 1 is informed of $k$  and Player 2 is informed of $\ell$.
\item Then, the game $(S,T,G^{k\ell})$ is played, i.e. Player 1 chooses $s\in S$, Player 2 chooses $t\in T$ (both choices being independent and simultaneous), 
and the payoff is $G^{k\ell}(s,t)$.
\end{itemize}
In normal-form, $\G(\pi)$ is given by a triplet $(S^K,T^L,\ga(\pi,\,\cdot\,))$ where: 
\[\ga(\pi,\hat{s},\hat{t}):=\sum_{(k,\ell) \in K\times L} \pi^{k \ell} G^{k\ell} (\hat{s}^k,\hat{t}^\ell).\]
Let $v^\pm(\pi)$ denote the maxmin and minmax of $\G(\pi)$. 

In order to apply the duality techniques described above, one needs to reformulate the primal game $\G(\pi)$ in a slightly different way. Let $p\in \De(K)$ and $Q\in \De(L)^K$ denote, respectively, the marginal of $\pi$ on $K$ and its matrix of conditional probabilities, so that:
$$\pi=p\otimes Q,\quad \left(\text{ i.e. } \pi^{k\ell}=p^kQ(\ell\,|\,k),\quad \forall (k,\ell)\in K\times L\right)$$
From the perspective of Player 2, the game with incomplete information on both sides $\G(\pi)$, can be seen as a game with incomplete information on one side, where Player 1 is the informed player.
\paragraph{The first primal game $\G_Q(p)$.} Let $Q\in \De(L)^K$ be fixed. For any $p\in \De(K)$, consider the following game:
\begin{itemize}
 \item Before the play, $k\in K$ is chosen according to $p$ and told to Player $1$. 
\item Then, the game $(S,T^L,G_Q^{k})$ is played, i.e. Player 1 chooses $s\in S$, Player 2 chooses $t\in T^L$ (both choices being independent and simultaneous), and the payoff is:
$$ G_Q^k(s,\hat{t}):=\sum_{\ell \in L}Q(\ell |k) G^{k\ell}(s, \hat{t}^\ell).$$
\end{itemize}
The sets of strategies are thus
  $S^K$ and $T^L$ and the payoff function is given by: $$\ga(p,\hat{s},\hat{t}):=\sum_{k \in K}p^k G_Q^k(\hat{s}^k,\hat{t}).$$ 
The maxmin and minmax are denoted respectively by $v_Q^-(p)$ and $v_Q^+(p)$, and the maps $v_Q^\pm(p)$ are concave and continuous. \\

The sets $S$ and $T^L$ are convex and $G^k_Q$ is bi-linear and bounded for all $k\in K$, so that a dual game can be defined, like in Section \ref{oneside}. 
 
\paragraph{The first dual game $\D[\G_Q](x)$.} As before, to every $x\in \RR^K$ corresponds a dual game. The sets of strategies are $\De(K)\times S^K$ and $T^L$ and the payoff function is given by:
$$h_Q[x](p,\hat{s},\hat{t}):=\sum_{k\in K} p^k G_Q^k(\hat{s}^k,\hat{t})-\langle p,x\rangle.$$ 
The maxmin and minmax functions $x\mapsto w_Q^\pm(x)$ are convex and continuous.

Theorem \ref{th1} and Corollary \ref{th2} can thus be restated accordingly. 

\begin{theoreme}\label{ddu} 
For each $(p,x)\in \De(K)\times \RR^K$ and $Q\in \De(L)^K$ one has:
$$w_Q^{\pm}(x)=(v_Q^\pm)^\sharp(x)\quad \text{ and } \quad 
 v_Q^{\pm}(p)=(w_Q^\pm)^\flat(p).$$
\end{theoreme} 
\begin{corollaire}\label{corodu} Let $(p,Q)\in \De(K)\times \De(L)^K$ and $\pi:=p\otimes Q$. 
Let $x\in \partial^+ v_Q^+(p)$, and let $\hat{t}_\ep\in T^L$ be an $\ep$-optimal strategy of Player 2 in $\D[\G_Q](x)$. Then, $\hat{t}_\ep$ is $\ep$-optimal for Player 2 in $\G(\pi)$.
\end{corollaire}

\paragraph{The second primal game $\G_P(q)$ and the second dual game $\D[\G_P](y)$.} 
In games with incomplete information on both sides, the two players have similar roles. 
Thus, by expressing the primal game $\G(\pi)$ from the perspective of Player 1, one similarly defines a primal game $\G_P(q)$ for Player 1 and the corresponding  dual game $\D[G_P](y)$, for all $(q,P)\in \De(L)\times \De(K)^L$ and $y\in \RR^L$. Analogue versions of Theorem \ref{ddu} and Corollary \ref{corodu} can thus be obtained.

\section{Repeated games with incomplete information}\label{repgames}
\subsection{Preliminaries}
Let $K,L,I,J$ be finite sets. For any $(k,\ell)\in K\times L$, let $G^{k\ell}=(G^{k\ell}(i,j))_{(i,j)\in I\times J}$ be an $I\times J$ matrix. 
A repeated game with incomplete information is described by the finite collection of matrix games 
$\{G^{k\ell}, \, (k,\ell) \in K\times L\}$ and a probability $\pi\in \De(K\times L)$. It is played as follows:
\begin{itemize}
\item   A pair of parameters $(k,\ell)\in K\times L$ is drawn according to $\pi\in \De(K\times L)$.
Player $1$ is informed of $k$, Player 2 is informed of $\ell$. 
\item 
Then, the game $G^{k\ell}$ is played repeatedly: at each stage $m\geq 1$, knowing the past actions, the players choose $(i_m,j_m)\in I\times J$ and a stage-payoff $G^{k\ell}(i_m,j_m)$ is produced (though not observed).
\end{itemize} 
The payoff of Player 1 is the expectation of
$\sum_{m\geq 1}\theta_m G^{k\ell}(i_m,j_m)$, for some given $\theta\in\De(\NN^*)$ that is known to both players. The payoff of Player 2 is the opposite amount.

We denote this game by $\G_\theta(\pi)$.

\paragraph{Strategies.} For each $m\geq 1$, let $\mathcal{H}_m:=(I\times J)^{m-1}$. The information available to the players at stage $m\geq 1$ is given by $(k,h_m)$ and $(\ell,h_m)$ for some $(k,\ell,h_m)\in K\times L\times \mathcal{H}_m$. Thus, a strategy of Player 1 is a function $\sh:K\times \mathcal{H}\to \De(I)$, where 
$\mathcal{H}:=\bigcup_{m\geq 1} \mathcal{H}_m$, and a strategy of Player 2 is a function $\hat{t}:L\times \mathcal{H}\to \De(J)$.
The set of strategies are denoted, respectively, by $S^K$ and $T^L$. 

\paragraph{Payoff and values.} A pair of strategies $(\sh,\hat{t})\in S^K\times T^L$ and an initial probability  $\pi\in \De(K\times L)$ induce a unique probability over $K\times L \times (I\times J)^{\NN^*}$ on the $\si$-algebra generated by the cylinders, denoted by 
$\PP^\pi_{\sh,\hat{t}}$. One can then write the game $\G_\theta(\pi)$ in normal-form, i.e. $\G_\theta(\pi)=(S^K,T^L,\ga_\theta(\pi,\,\cdot\,))$ where: 
\[ \ga_\theta(\pi,\sh,\hat{t}):=\EE_{\sh,\hat{t}}^\pi \left[\sum\nolimits_{m\geq 1}\theta_m G^{k\ell}(i_m,j_m)\right],\]
and where $\EE^\pi_{\sh,\hat{t}}$ is the expectation with respect to  $\PP^\pi_{\sh,\hat{t}}$. \\

The following result is well-known \footnote{One may apply Sion's minmax theorem to the game in mixed strategies when pure strategies are endowed with the product topology, and then apply Kuhn's theorem to deduce the result. See e.g. chapter 3 and appendix A in \cite{sorin02} where the same method is applied in the discounted case.}, and we omit its proof. 

\begin{lemme}\label{val} 
For any $\pi\in\De(K\times L)$, the game $\G_\theta(\pi)$ has a value, i.e. 
$$v_\theta(\pi):= \sup_{\sh\in S^K} \inf_{\hat{t}\in T^L}\ga_\theta(\pi,\sh,\hat{t})=\inf_{\hat{t}\in T^L}\sup_{\sh\in S^K} \ga_\theta(\pi,\sh,\hat{t}).$$
Moreover, both players have $0$-optimal strategies.
\end{lemme}

The aim of this paper is to provide an explicit construction for a couple of $0$-optimal strategies. Recall that, as the two players have similar roles, we will focus on the construction for Player 2 only. For this reason, from now on the function  $v_\theta:\De(K\times L)\to \RR$ will be expressed in the following equivalent manner:
$$v_\theta(p,Q):=v_\theta(p\otimes Q), \quad \forall (p,Q)\in \De(K)\times \De(L)^K,$$
which is more convenient for studying the game from the perspective of Player 2. 
Let us start by recalling an important result, the so-called \emph{primal recursive formula}, which expresses the values of the repeated game with incomplete information $v_\theta(p,Q)$ in terms of the values of the \emph{continuation game}, that is, the sub-game that the players are facing after the first stage.

\subsection{Primal recursive formula}\label{rec_f}
Like in Section \ref{dualide}, for each $(p,Q)\in \De(K)\times \De(L)^K$, let $\G_\theta(p,Q)$ denote the repeated game of incomplete information on both sides $\G_\theta(\pi)$, expressed from the point of view of Player 2. 

The aim of this section is to provide a recursive formula satisfied by the values $v_\theta(p,Q)$. 
The following specific notation will be used to express this result. 
\begin{itemize}
\item $(\si,\tau)\in \De(I)^K\times \De(J)^L$ denotes a pair of strategies for the first stage of the game. 
\item For any $(p,Q)\in \De(K)\times \De(L)^K$, $(\si,\tau)\in \De(I)^K\times \De(J)^L$ and $(i,j) \in I\times J$
define $p_{ij}\in \De(K)$ and $Q_{j}\in \De(L)^K$ by setting, for all $(k,\ell)$: 
\begin{equation*}\label{pQdefs} 
p^k_{ij}:=\PP^{pQ}_{\si\tau}(k\, |\, i,j) \quad \text{ and }\quad Q_j(\ell\,|\,k):=\PP^{pQ}_{\si\tau}(\ell\,|\,k,j),
\end{equation*} 
where $\PP_{\si\tau}^{pQ}$ is the probability over $K\times L\times I\times J$ induced by $(\si,\tau,p,Q)$.
\end{itemize}

The following easy result is important, as in particular it shows that 
Player 2 can compute (and, in fact, controls) the matrix of conditional probabilities $Q_j\in \De(L)^K$ for all $j\in J$. 

\begin{lemme}\label{pij} For all $(k,\ell,i,j)$ one has:
$$ Q_{j}(\ell\,|\,k)= \frac{Q(\ell|k)\tau^\ell(j)}{\sum\nolimits_{\ell'\in L} Q(\ell'| k)\tau^{\ell'}(j)}\quad \text{ and }\quad  \PP^{pQ}_{\si\tau}(k,\ell\,|\,i,j)=p_{ij}^k Q_j(\ell\,|\,k).
$$
\end{lemme}

\begin{proof} For any $(k,j)$ such that $\PP_{\si\tau}^{pQ} (k,j)>0$, a direct computation gives: 
\[\PP_{\si\tau}^{pQ} (\ell\,|\,k,j)=\frac{p^k Q(\ell|k)\tau^\ell(j)}{p^k \sum_{\ell'\in L}Q(\ell'|k)\tau^{\ell'}(j)}=\frac{Q(\ell|k)\tau^\ell(j)}{\sum_{\ell'\in L}Q(\ell'|k)\tau^{\ell'}(j)}.\]
so that the first relation holds. Similarly, for any $(k,i,j)$ such that $\PP_{\si\tau}^{pQ} (k,i,j)>0$ one obtains: $$\PP_{\si\tau}^{pQ} (\ell\,|\,k,i,j)=\PP_{\si\tau}^{pQ} (\ell\,|\,k,j).$$ 
For any $(i,j)$ such that $\PP_{\si\tau}^{pQ} (i,j)>0$, 
disintegration gives then the second relation: 
\[\PP^{pQ}_{\si\tau}(k,\ell\,|\,i,j)=\PP_{\si\tau}^{pQ} (k\,|\, i,j)\PP_{\si\tau}^{pQ} (\ell\,|\,k,i,j)=p^k_{ij}Q_j(\ell\,|\,k)
.\]
\end{proof} 

We are now ready to state the so-called primal recursive formula, due to 
Mertens and Zamir \cite[Section 3]{MZ71}. For convenience, we provide a direct and shorter proof here. 
\begin{proposition}[Primal recursive formula]\label{recprim} 
For any $(p,Q)\in \De(K)\times \De(L)^K$ and 
 $\theta\in \De(\NN^*)$: 
\begin{eqnarray*} 
v_\theta(p,Q)&=& \max_{\si\in \De(I)^K}\; \min_{\tau\in \De(J)^L}\; \left\{
\theta_1G^{pQ}_{\si\tau}+(1-\theta_1)\sum_{(i,j)\in I\times J}\PP^{pQ}_{\si\tau}(i,j)v_{\theta^+}(p_{ij}, Q_j)\right\},\\
&=& \min_{\tau\in \De(J)^L}\;\max_{\si\in \De(I)^K}\; \left\{
\theta_1G^{pQ}_{\si\tau}+(1-\theta_1)\sum_{(i,j)\in I\times J}\PP^{pQ}_{\si\tau}(i,j)v_{\theta^+}(p_{ij}, Q_j)\right\}.
\end{eqnarray*}
\end{proposition}
\begin{proof}
Consider the maxmin. Let $\sh$ be a strategy of Player 1. At the first stage, the information available to Player 1 is $k$, so that $\si:=(\sh(k))_k \in \De(I)^K$ represents the strategy of Player 1 at the first stage. Similarly, the information available to Player 1 at the second stage is $(k,i,j)$ for some couple of actions $(i,j)$ played at the first stage, so that $\sh^+=(\sh(k,i,j))_{k,i,j}$ represents the strategy of Player 1 at the second stage. One can then write $\sh=(\si,\sh^+)$. Similarly, a strategy $\hat{t}$ of Player 2 can be written as $\hat{t}=(\tau,\hat{t}^+)$ where $\tau \in \De(J)^L$ and $\hat{t}^+=(\hat{t}(\ell,i,j))_{\ell,i,j}$.

For each $(i,j)$, let $\hat{t}^+_{ij}$ be a best reply of Player 2 to the strategy $\sh^+_{ij}:=(\sh(k,i,j))_{k}$ in the so-called continuation game, i.e.  
 $\G_{\theta^+}(p_{ij},Q_j)$. Then, for all $(\si,\tau)\in \De(I)^K\times \De(J)^L$ one has:
\begin{equation*}
\ga_\theta\left(p\otimes Q,\sh,\hat{t}\right)\leq \theta_1 G^{pQ}_{\si\tau} + (1-\theta_1)\sum\nolimits_{(i,j)\in I\times J}\PP^{pQ}_{\si\tau}(i,j)v_{\theta^+}(p_{ij},Q_j).
\end{equation*}
Player $1$ can still maximize over his own first-stage strategy, and Player 2 can again play a best reply. Hence:
\begin{eqnarray*}\label{rec-}
v_\theta(p,Q)&=&\max_{(\si,\sh^+)}\min_{(\tau,\hat{t}^+)} \ga_\theta(p\otimes Q,\sh,\hat{t}),\\ 
&\leq & \max_{\si\in \De(I)^K}\min_{\tau\in \De(J)^L} \left \{\theta_1 G^{pQ}_{\si\tau} + (1-\theta_1)\sum\nolimits_{(i,j)\in I\times J}\PP^{pQ}_{\si\tau}(i,j)v_{\theta^+}(p_{ij},Q_j)\right\}.
\end{eqnarray*}
Reversing the roles of the players one obtains, symmetrically:
\begin{equation}\label{rec+}
v_\theta(p,Q)\geq \min_{\tau\in \De(J)^L}\max_{\si\in \De(I)^K}\left\{\theta_1 G^{pQ}_{\si\tau}+(1-\theta_1)\sum\nolimits_{(i,j)\in I\times J}\PP^{pQ}_{\si\tau}(i,j)v_{\theta^+}(p_{ij},Q_j)\right\}.
\end{equation}
The result follows then from the inequality ``maxmin $\leq$ minmax''.
\end{proof}

\paragraph{Comments.} Proposition \ref{recprim} provides a recursive formula satisfied by the values, from the perspective of Player 2. Similarly, one can obtain a recursive formula satisfied by the values from the perspective of Player 1, expressing $v_\theta(q,P):=v_\theta(q\otimes P)$ in terms of $v_{\theta^+}(q_{ij},P_i)$ for any $(q,P)\in \De(L)\times \De(K)^L$ and suitably defined conditional probabilities $(q_{ij},P_i)\in \De(L)\times \De(K)^L$ for all $(i,j)$. However, neither of these recursive formulae can be used by the players as both probability distributions $p_{ij}$ and $q_{ij}$ depend 
on the first-stage strategies of the two players. The situation contrasts with the case of repeated games with incomplete information on one side (that is, when $L$ is a singleton), where Player 1 observes and controls the conditional probability $p_i\in \De(K)$, so that the primal recursive formula 
provides an explicit and recursive manner to obtain an optimal strategy for Player 1  (see \cite[Section 3]{sorin02}).

\begin{remarque} The sequence of weights $\theta^+$ is not defined when $\theta_1=1$, so that neither is the value function $v_{\theta^+}$. This, however, does not matter as the primal formula contains the term $(1-\theta_1)v_{\theta^+}$, which is $0$.\end{remarque}

\section{Dual recursive formula}\label{dualrep}
The aim of this section is to prove the \emph{dual recursive formula}, stated in Theorem \ref{thmmain}, and to deduce  an explicit construction of an optimal strategy for Player 2 in $\G_\theta(\pi)$.

\subsection{The dual game}
Consider the game $\G_\theta(\pi)$ described in Section \ref{repgames} from the point of view of Player 2. For any fixed matrix of conditionals $Q\in\De(L)^K$ and any $k\in K$, consider the collection of games $\{\G_\theta(\de_k\otimes Q), k\in K\}$, where $\de_k\in \De(K)$ is the Dirac mass, i.e. $\de_k(k)=1$ and $\de_{k}(k')=0$ for all $k'\neq k$.  
Let $S$ and $T^L$ denote, respectively, the common sets of strategies of Player 1 and 2 in each of these games. 
These sets are convex and the payoff functions $(s,\hat{t})\mapsto \ga_\theta(\de_k\otimes Q,s,\hat{t})$ are bi-linear and bounded. Hence, like in Section \ref{applications}, one can define the corresponding dual game $\D[\G_\theta](x,Q)$ for Player 2, for any $x\in \RR^K$. 

\paragraph{The dual game $\D[\G_\theta](x,Q)$.} By construction, the sets of strategies of this game 
 are $\De(K)\times S^K$ and $T^L$, and the payoff function is  
given by:
\[h_\theta[x,Q](p,\sh,\hat{t})=\ga_\theta\left(p\otimes Q,\sh,\hat{t}\right)-\langle p, x\rangle.\]
\textcolor{black}{By Lemma \ref{val} and Theorem \ref{ddu}, this game has a value $w_{\theta}(x,Q)$, i.e.
\begin{eqnarray*}
 w_{\theta}(x,Q)&=&\max_{(p,\sh)\in \De(K)\times S^K} \min_{\hat{t}\in T^L} h_\theta[x,Q](p,\sh,\hat{t}),\\
 &=&\min_{\hat{t}\in T^L} \max_{(p,\sh)\in \De(K)\times S^K}h_\theta[x,Q](p,\sh,\hat{t}).
\end{eqnarray*}}
and the mappings $x\mapsto w_\theta(x,Q)$ are convex and continuous.\\

The following notation will be used in the proof of the dual recursive formula:

\begin{itemize}
\item $\PP^Q_{\mu\tau}$ is the unique probability on $K\times L\times I\times J$ induced by $\mu\in \De(K\times I)$, $Q\in \De(L)^K$ and $\tau\in\De(J)^L$, i.e. $\PP^Q_{\mu\tau}(k,\ell,i,j)=\mu(k,i)Q(\ell\,|\,k)\tau^\ell(j)$.
\item $G^{Q}_{\mu\tau}$ denotes the expectation of $G^{k\ell}(i,j)$ with respect to the probability $\PP_{\mu\tau}^{Q}$. Note that $G^{kQ}_{i\tau}$ stands for $G^{Q}_{\mu\tau}$ for the particular case where $\mu=\de_{(k,i)}$ for some $(k,i)\in K\times I$.
\item $\PP^{kQ}_{\tau}(j):= \sum_{\ell \in L} Q(\ell |k) \tau^\ell(j)$ for any 
$k\in K$, $Q\in \De(L)^K$ and $\tau\in \De(J)^L$. 
\item $\|G\|:=\max_{k,\ell,i,j} |G^{k\ell}(i,j)|$.
\item $B$ denotes the following set $\{ x\in \RR^K \,|\, \|x\|_\infty \leq \|G\|\}$. 
 \end{itemize}

\subsection{Proof of Theorem \ref{thmmain}}

Let us recall the statement of the theorem:\\

 \emph{For all $(x,Q)\in \RR^K \times \De(L)^K$ and $\theta\in \De(\NN^*)$ one has: 
\begin{eqnarray*}
w_{\theta}(x,Q)&=&\min_{\substack{\tau\in \De(J)^L\\ (x_{ij})_{ij}\in (\RR^K)^{I\times J}}} \max_{(k,i)\in K \times I} \left\{ 
\theta_1 G^{kQ}_{i \tau}+(1-\theta_1)\sum\nolimits_{j\in J}\PP^{kQ}_{\tau}(j) \left( w_{\theta^+}(x_{ij},Q_j)
+ x_{ij}^k \right) -x^k \right\}.
\end{eqnarray*}}
Furthermore, we will prove that the minimum in $(x_{ij})_{ij}$ is reached in the set $B^{I\times J}$.

\begin{remarque} Again, there is no need in defining the value function $w_{\theta^+}$ when $\theta_1=1$, since in this case the term $(1-\theta_1)w_{\theta^+}$ is equal to $0$ by convention.\end{remarque}

\begin{proof}On the one hand, by Proposition \ref{recprim} one has the primal recursive formula:
$$v_\theta(p, Q)=\max_{\si\in \De(I)^K}\min_{\tau\in \De(J)^L}\left\{ 
\theta_1G^{pQ}_{\si\tau}+(1-\theta_1)\sum_{(i,j)\in I\times J}\PP^{pQ}_{\si\tau}(i,j)v_{\theta^+}(p_{ij}, Q_j)\right\}, $$
where $p_{ij}\in \De(K)$ and $Q_j\in \De(L)^K$ are defined in \eqref{pQdefs} and where, by Lemma \ref{pij}, 
$Q_j$ does not depend on $(p,\si)$. 
On the other hand, by the duality results presented in Section \ref{applications}, namely Theorem \ref{ddu}, one has:  
$$w_\theta(x,Q)=\max_{p\in \De(K)}v_\theta(p, Q)-\langle p,x\rangle.$$
Replacing $v_\theta(p, Q)$ by its expression in the primal recursive formula one obtains: 
\begin{eqnarray*} w_\theta(x,Q)&=&\max_{p\in \De(K)}\left\{ \max_{\si\in \De(I)^K}\min_{\tau\in \De(J)^L}\left\{ 
\theta_1G^{pQ}_{\si\tau}+(1-\theta_1)\sum_{(i,j)\in I\times J}\PP^{pQ}_{\si\tau}(i,j)v_{\theta^+}(p_{ij}, Q_j)\right\}-\langle p,x\rangle\right\},\\ 
&=& \max_{\mu\in \De(K\times I)} \min_{\tau\in \De(J)^L} \left\{ 
\theta_1G^{Q}_{\mu \tau}+(1-\theta_1)\sum_{(i,j)\in I\times J}\PP^{Q}_{\mu \tau}(i,j)v_{\theta^+}(\mu_{ij}, Q_j)-\langle \mu^K, x\rangle \right\}.
\end{eqnarray*}
where 
$\mu^K$ is the marginal of $\mu$ on $K$ and $\mu_{ij}$ is the conditional on $K$ given $(i,j)$, i.e. 
$$\mu^k_{ij}=\PP^{Q}_{\mu\tau}(k\,|\,i,j),\quad \forall (k,i,j).$$
Consider the one-shot game with action sets $\De(K\times I)$ and $\De(J)^L$ and payoff function:
$$F[\theta,x,Q](\mu,\tau):=\theta_1G^{Q}_{\mu \tau}+(1-\theta_1)\sum_{(i,j)\in I\times J}\PP^{Q}_{\mu\tau}(i,j)v_{\theta^+}(\mu_{ij}, Q_j)-\langle \mu^K, x\rangle.$$
Clearly, $F[\theta,x,Q]$ is continuous on the compact set $\De(K\times I)\times \De(J)^L$, its first and last terms are linear in $\mu$ and $\tau$ and, as we have already shown, one has:
 $$w_\theta(x,Q)= \max_{\mu\in \De(K\times I)} \min_{\tau\in \De(J)^L} F[\theta,x,Q](\mu,\tau).$$
Therefore, one can apply Sion's minmax theorem \cite{sion58} as soon as we prove that the following function is concave-convex: 
\begin{eqnarray*}f(\mu,\tau)&:=&\sum_{(i,j)\in I\times J}\PP^{Q}_{\mu\tau}(i,j)v_{\theta^+}(\mu_{ij}, Q_j).
\end{eqnarray*}
First, let us recall that from Theorem \ref{ddu}, the following  relation holds for all $(i,j)$: 
\begin{equation}\label{eqfenchel}
v_{\theta^+}(\mu_{ij}, Q_j) = \inf_{x_{ij} \in \RR^K} w_{\theta^+}(x_{ij},Q_j) + \langle \mu_{ij}, x_{ij} \rangle
\end{equation} 
Recall that any concave function $\varphi$ on $\Delta(K)$ which is $\|G\|$-Lipschitz with respect to the norm $\|.\|_1$ can be extended to a concave Lipschitz function $\tilde \varphi$ on the whole space $\RR^K$ having the same Lipschitz constant, by defining $\tilde \varphi(x)=\sup_{y \in \RR^K} \{ \varphi(y)- \|G\|\|y-x\|_1 \}$. By construction, $\tilde \varphi$ admits super-gradients at every point, which belong to the compact convex set $B:=\{x\in \RR^K\,|\, \|x\|_\infty\leq \|G\|\}$. Since $\varphi=\tilde\varphi$ on $\Delta(K)$, for all $p\in \Delta(K)$, the super-gradients of $\tilde\varphi$ at $p$ are super-gradients of $\varphi$ at $p$, and therefore $\varphi$ admits super-gradients in $B$ at every point of $\Delta(K)$.  

According to Fenchel's lemma, the set of  minimisers of the right-hand side of \eqref{eqfenchel} is exactly the set of super-gradients of the concave mapping $p'\mapsto v_{\theta^+}(p',Q_j)$ at $p'=\mu_{ij}$. Since this mapping is $\|G\|$-Lipschitz with respect to the norm $\|.\|_1$ on its domain $\Delta(K)$, it admits a super-gradient in the compact convex set $B$ at $\mu_{ij}$, which is therefore a minimiser of the right-hand side of \eqref{eqfenchel}.
We deduce that
\begin{equation}\label{eqfenchel2}
v_{\theta^+}(\mu_{ij}, Q_j) = \min_{x_{ij} \in B} w_{\theta^+}(x_{ij},Q_j) + \langle \mu_{ij}, x_{ij} \rangle
\end{equation} 

Hence, replacing this expression, one can write: 
\begin{align*}
f(\mu,\tau)& = \inf_{(x_{ij})_{ij} \in B^{I\times J}} \sum_{(i,j)\in I\times J}\PP^{Q}_{\mu\tau}(i,j) \left(w_{\theta^+}(x_{ij},Q_j) + \langle \mu_{ij}, x_{ij} \rangle \right), \\
&=\inf_{(x_{ij})_{ij} \in B^{I\times J}} \left( \sum_{(i,j)\in I\times J}\PP^{Q}_{\mu\tau}(i,j)w_{\theta^+}(x_{ij},Q_j) + \sum_{(i,j,k)\in I\times J\times K} \PP^{Q}_{\mu\tau}(i,j) \mu_{ij}^k x_{ij}^k \right), \\
&=\inf_{(x_{ij})_{ij} \in B^{I\times J}} \left( \sum_{(i,j)\in I\times J}\PP^{Q}_{\mu\tau}(i,j)w_{\theta^+}(x_{ij},Q_j) + \sum_{(i,j,k)\in I\times J\times K} \PP^{Q}_{\mu\tau}(i,j,k)x_{ij}^k \right).
\end{align*}
Since $\mu \rightarrow  \PP^{Q}_{\mu\tau}$ is affine, we deduce from the above expression that $\mu \rightarrow f(\mu,\tau)$ is concave as an infimum of affine functions.

To prove the convexity of $\tau\mapsto f(\mu,\tau)$ we will consider the primal game $\G_\theta(\pi)$ from the point of view of Player 1. For any $q\in \De(L)$ and $P\in \De(K)^L$, denote this game by $\G_\theta(q,P)$ and let $v_\theta(q,P):=v_\theta(q\otimes P)$ denote its value. Using this notation, for each $(i,j)$ one can write: 
$$v_{\theta^+}(\mu_{ij}, Q_j)={v}_{\theta^+}(q_{ij}, P_i),$$
where $q_{ij}\in \De(L)$ and $P_i\in \De(K)^L$. In particular, $P_i$ is independent from $j$ and $\tau$ (just like $Q_j$ does not depend neither on $i$ nor $\si$). Explicitly, for all $(k,\ell,i,j)$ one has:
$$q_{ij}^\ell= \PP^Q_{\mu\tau}(\ell\,|\,i,j)\quad \text{ and }\quad  P_{i}(k\,|\,\ell)= \PP^Q_{\mu\tau}(k\,|\,\ell, i)=
\frac{\mu(k,i)Q(\ell\,|\,k)}{\sum_{k'\in K} \mu(k',i)Q(\ell\,|\,k')}.$$ 
Use the duality techniques of Section \ref{applications} to define, for each $y\in \RR^L$, the dual game $\D[\G_\theta](y,P)$, and denote its value by $w_\theta(y,P):=\sup_{q\in \De(L)}v_\theta(q,P)-\langle q,y\rangle$.
For each $(i,j)$ one then has:
\[v_{\theta^+}(q_{ij}\otimes P_i) = \sup_{y_{ij} \in \RR^L} {w}_{\theta^+}(y_{ij},P_i) -\langle q_{ij}, y_{ij} \rangle ,\]
so that: 
\begin{align*}
f(\mu,\tau)& = \sup_{(y_{ij})_{ij} \in (\RR^L)^{I\times J}} \sum_{(i,j)\in I\times J}\PP^{Q}_{\mu\tau}(i,j) \left({w}_{\theta^+}(y_{ij},P_i) -\langle q_{ij}, y_{ij} \rangle\right), \\
&= \sup_{(y_{ij})_{ij} \in (\RR^L)^{I\times J}} \left( \sum_{(i,j)\in I\times J}\PP^{Q}_{\mu\tau}(i,j){w}_{\theta^+}(y_{ij},P_i) - \sum_{(i,j,\ell)\in I\times J\times L} \PP^{Q}_{\mu\tau}(i,j) q_{ij}^\ell y_{ij}^\ell \right), \\
&=\sup_{(y_{ij})_{ij} \in (\RR^L)^{I\times J}} \left( \sum_{(i,j)\in I\times J}\PP^{Q}_{\mu\tau}(i,j){w}_{\theta^+}(y_{ij},P_i) - \sum_{(i,j,\ell)\in I\times J\times L} \PP^{Q}_{\mu\tau}(i,j,\ell)y_{ij}^\ell \right).
\end{align*}
The mappings $\tau \rightarrow  \PP^{Q}_{\mu\tau}(i,j)$ and $\tau \rightarrow  \PP^{Q}_{\mu\tau}(i,j,\ell)$ being affine, the previous expression shows that $\tau \rightarrow f(\mu,\tau)$ is convex, as a supremum of affine functions.

\noindent Therefore, one can indeed apply Sion's minmax theorem. Exchanging the maximum and the minimum one obtains:
 $$w_\theta(x,Q)= \displaystyle  \min_{\tau\in \De(J)^L} \max_{\mu\in \De(K\times I)}F[\theta,x,Q](\mu,\tau).$$ 
Replacing now $v_{\theta^+}(\mu_{ij}, Q_j)$ with $\min_{x_{ij}\in B} w_{\theta^+}(x_{ij},Q_j)+\langle \mu_{ij},x_{ij}\rangle$ gives then: 
\begin{align*} 
w_\theta(x,Q)&=  
 \min_{\tau\in \De(J)^L} \max_{\mu\in \De(K\times I)}\min_{(x_{ij})_{ij}\in B^{I\times J}} \\
&  \left\{ 
\theta_1G^{Q}_{\mu \tau}+(1-\theta_1)\sum\nolimits_{(i,j)\in I\times J}\PP^{Q}_{\mu \tau}(i,j)\left(w_{\theta^+}(x_{ij},Q_j)+\langle \mu_{ij},x_{ij}\rangle\right) -\langle \mu^K, x\rangle \right\}.\end{align*}
Again, in order to apply Sion's minmax theorem to exchange the order of the maximum and the infimum 
one needs to check that the mapping $(\mu,\xx)\mapsto g(\mu,\xx)$ is concave-convex, where $\xx:=(x_{ij})_{ij}\in B^{I\times J}$ and: 
\begin{align*}
g(\mu,\xx):&=\sum_{(i,j)\in I\times J}\PP^{Q}_{\mu \tau}(i,j) \left(w_{\theta^+}(x_{ij},Q_j)+\langle \mu_{ij},x_{ij}\rangle\right), \\
&=\sum_{(i,j)\in I\times J}\PP^{Q}_{\mu \tau}(i,j)w_{\theta^+}(x_{ij},Q_j)+ \sum_{(i,j,k)\in I\times J\times K} \PP^{Q}_{\mu \tau}(i,j,k) x_{ij}^k. 
\end{align*}
This property follows from the fact that the mappings $\mu \rightarrow   \PP^{Q}_{\mu \tau}(i,j)$ and $\mu \rightarrow   \PP^{Q}_{\mu \tau}(i,j,k)$ are affine, and that the map $x \rightarrow w_{\theta^+}(x,Q_j)$ is convex. 
We thus obtain:
\begin{align*} 
w_\theta(x,Q) &=  
 \min_{\tau\in \De(J)^L} \min_{(x_{ij})_{ij}\in B^{I\times J}}\max_{\mu\in \De(K\times I)} 
 \\
&  \left\{ 
\theta_1G^{Q}_{\mu \tau}+(1-\theta_1)\sum_{i,j}\PP^{Q}_{\mu \tau}(i,j) w_{\theta^+}(x_{ij},Q_j)
+(1-\theta_1) \sum_{i,k} \mu(i,k) \sum_{j} \PP^Q_{\mu \tau}(j|i,k) x_{ij}^k  -\langle \mu^K, x\rangle \right\}.
\end{align*}
Since the expression above is affine with respect to $\mu$, we can consider without loss of generality
the maxima at extreme points:
\begin{align*} 
w_\theta(x,Q) &=  
 \min_{\tau\in \De(J)^L} \min_{(x_{ij})_{ij}\in B^{I\times J}}\max_{(k,i) \in K\times I} 
 \\
&  \left\{ 
\theta_1G^{kQ}_{i \tau}+(1-\theta_1)\sum\nolimits_{j}\PP^{kQ}_{\tau}(j) \left( w_{\theta^+}(x_{ij},Q_j)
+ x_{ij}^k \right) -x^k \right\}.
\end{align*}
To conclude, note that the minimum over $B^{I\times J}$ can be replaced by a minimum over $(\RR^K)^{I\times J}$ since the above proof is still  valid if we replace $v_{\theta^+}(\mu_{ij}, Q_j)$ by the expression given by \eqref{eqfenchel} instead of the one given by \eqref{eqfenchel2}.
\end{proof}

\begin{remarque}\label{rem_proof} 
The proof of Theorem \ref{thmmain} follows the main lines of \cite{rosenberg98}, 
but there is a crucial point where an obstacle arises, namely in proving that the function 
$$f(\mu,\tau):=\sum_{(i,j)\in I\times J}\PP^{Q}_{\mu\tau}(i,j)v_{\theta^+}(\mu_{ij}, Q_j)$$
is concave-convex. Unlike the independent case, where the proof relies deeply on the fact 
that $(p,q)\mapsto v_{\theta^+}(p,q)$ is a concave-convex functions of independent probabilities $p\in \De(K)$ and $q\in \De(L)$,  the arguments of $v_{\theta^+}(\mu_{ij}, Q_j)$ are not independent from each other, and one thus needs to use the duality techniques for the dependent case, which are more sophisticated.
In this point, our proof diverges from the one in \cite{rosenberg98}.
\end{remarque}

\subsection{Construction of an optimal Markovian strategy}\label{markov}
In this section, we deduce from Theorem \ref{thmmain} and Corollary \ref{corodu} the construction of an optimal strategy for Player 2 in the game $\G_\theta(\pi)$. The strategy is Markovian, in the sense that it depends on the past history only through the (updated) variable $(x,Q,\theta)$.

Let $(p,Q)\in\De(K)\times \De(L)^K$ be such that $\pi=p\otimes Q$.
Let $x\in \partial^+ v_{\theta}(p,Q)$ be a super-gradient at $p$ of the mapping $v_\theta(\, \cdot \, , Q):\De(K)\to \RR$.

For all $(\theta',x',Q') \in \Delta(\NN^*)\times \RR^K\times \De(L)^K$, let us denote by $S_{\theta'}(x',Q')$ the set of minimizers in the dual recursive formula, i.e.
\[ S_{\theta'}(x',Q')=\text{argmin}_{(\tau,\xx) \in \De(J)^L \times \RR^{K\times I\times J}}\; h(\theta',x',Q')[\tau,\xx],\]
where $\xx=(x_{ij})_{ij}$ and where, using the notation of the previous sections:
\[ h(\theta',x',Q')[\tau,\xx]=\max_{(k,i) \in K\times I} \left\{ 
\theta'_1G^{kQ'}_{i \tau}+(1-\theta'_1)\sum\nolimits_{j\in J}\PP^{kQ'}_{\tau}(j) \left( w_{(\theta')^+}(x_{ij},Q'_j)
+ x_{ij}^k \right) -(x')^k \right\}.\]
An optimal strategy for Player 2 in $\G_\theta(\pi)$ can be constructed, recursively, as follows.\\

\noindent \textbf{Case 1.} If $\theta_1=1$, play $\tau\in \De(J)^L$ which is optimal in the formula:
$$w_1(x,Q)=\min_{\tau\in \De(J)^L} \max_{(k,i)\in K \times I} \left\{ 
\theta_1 G^{kQ}_{i \tau}-x^k \right\}.
$$
\noindent \textbf{Case 2.} If $\theta_1<1$, plays as follows:
\begin{itemize}
\item[-] Compute $(\tau, \xx) \in S_\theta(x,Q)$ optimal in the dual recursive formula at $(x,Q,\theta)$. 
\item[-] Choose $j\in J$ with probability $\tau^\ell$, where 
$\ell\in L$ is Player 2's private type.
\item[-] Observe $i\in I$ and update the triplet $(x,Q,\theta)$ to $(x_{ij},Q_j,\theta^+)$. 
\item[-] Play optimally in $\D[\G_{\theta^+}](x_{ij}, Q_j)$.
\end{itemize}
This strategy is optimal in the dual game $\D[\G_{\theta}](x, Q)$, thanks to Theorem \ref{thmmain}. 
By Corollary \ref{corodu} and the choice of $(x,Q)$, the strategy is also optimal in the game $\G_\theta(\pi)$. 
Furthermore, it is \emph{Markovian}. Indeed, at every stage $m\geq 1$, the mixed action of Player $2$  at this stage $\tau_m \in \De(J)^L$ depends only on a triplet of variables $(x^{(m-1)},Q^{(m-1)},\theta^{(m-1)})\in \RR^K\times \De(L)^K\times \De(\NN^*)$
constructed recursively as follows. For $m=0$ set: 
$$\begin{cases}
x^{(0)}:=x\\
Q^{(0)}:=Q\\
\theta^{(0)}:=\theta\, . \end{cases}
$$
For all $m\geq 1$, if $\theta^{(m-1)}_1<1$, compute $(\tau_m,\xx^{(m)})\in S_{\theta^{(m-1)}}(x^{(m-1)},Q^{(m-1)})$ 
and set:
$$\begin{cases}
x^{(m)}:=(\xx^{(m)})_{i_m j_m}\\
Q^{(m)}(\ell|k)= \frac{Q^{(m-1)}(\ell|k)\tau_m^\ell(j_m)}{\sum\nolimits_{\ell'\in L} Q^{(m-1)}(\ell'| k)\tau_m^{\ell'}(j_m)}, & \forall (k,\ell)\in K\times L,\\ 
\theta^{(m)}_t: =\frac{\theta_{m+t}}{\sum_{t' > m}\theta_{t'}}, & \forall t\geq 1, 
\end{cases}$$
where $(i_m,j_m)$ is the pair of actions played at stage $m$.

\section{Comments and Extensions}

\subsection{Stochastic games with incomplete information}
Consider a game with incomplete information over a finite set of states $S$, where each state represents a different state of the world. Formally, let $G^{k\ell}:S\times I\times J\to \RR$ denote a payoff function for each pair of types $(k,\ell)\in K\times L$,  depending not only on the players' actions, but also on the state, and let $\rho:S\times I\times J\to \De(S)$ denote a transition kernel. For any $\theta\in \De(\NN^*)$, any $\pi\in \De(K\times L)$ and 
any initial state $s_1\in S$, the \emph{stochastic game with incomplete information on both sides}, denoted by $\G_\theta(\pi; s_1)$,
is played as follows: 
\begin{itemize}
\item First, a pair of parameters $(k,\ell)\in K\times L$ is drawn according to $\pi\in \De(K\times L)$.
Player $1$ is informed of $k$, Player 2 is informed of $\ell$. 
\item At each stage $m\geq 1$, knowing the current state $s_m\in S$  and knowing the past actions, the players choose actions $(i_m,j_m)\in I\times J$. A stage-payoff $G^{k\ell}(s_m,i_m,j_m)$ is produced (though not observed) and a new state $s_{m+1}\in S$ is drawn with the probability distribution $\rho(s_m,i_m,j_m)$. 
\end{itemize} 
The payoff of Player 1 is the expectation of
$\sum_{m\geq 1}\theta_m G^{k\ell}(s_m,i_m,j_m)$, while the payoff of Player 2 is the opposite amount. 

\begin{remarque} The case where the set of states $S=\{s\}$ is a singleton corresponds to repeated games with incomplete information on both sides. In this sense, stochastic games with incomplete information extend our previous model.  
\end{remarque}

Let $p\in \De(K)$ and $Q\in\De(L)^K$ be such that $\pi=p\times Q$ and let $s\in S$ be an initial state. The existence of the value for stochastic games with incomplete information is well-known, and we omit its proof.  
Let $v_\theta(\pi,s)$ and $w_\theta(x,Q,s)$ denote, respectively, the values of $\G_\theta(\pi; s)$ and of the corresponding first dual game. The dual recursive formula obtained in Theorem \ref{thmmain} can be extended, word for word, to stochastic games with incomplete information on both sides in the dependent case. 

\paragraph{Notation.} In the following result we use the notations introduced earlier. Moreover, for each $(Q,\tau)\in \De(L)^K\times \De(J)^L$ and $(s,k,i)\in S\times K\times I$, we set:
$$G^{kQ}_{i \tau}(s)=\sum_{(\ell,j)\in L\times J}Q(\ell|k)\tau^\ell(j)G^{k\ell}(s,i,j)\, .$$

\begin{theoreme}[Dual recursive formula]\label{thmmain2} For all $(x,Q,s)\in \RR^K \times \De(L)^K\times S$ and $\theta\in \De(\NN^*)$ one has:
\begin{eqnarray*}\label{eee3}
w_{\theta}(x,Q,s)&=&\min_{\tau\in \De(J)^L} \min_{(x_{ijs})_{ij}\in (\RR^K)^{I\times J\times S}} \max_{(k,i)\in K \times I} \\ & & \left\{ 
\theta_1 G^{kQ}_{i \tau}(s)+(1-\theta_1)\sum_{(s',j)\in S\times J}\PP^{kQ}_{\tau}(j)\rho(s'|s,i,j)\left( w_{\theta^+}(x_{ijs},Q_j,s')
+ x_{ijs}^k \right) -x^k \right\}.
\end{eqnarray*}

\end{theoreme}

\begin{corollaire}\label{cormain2}Player $2$ can construct an optimal strategy in $\G_\theta(\pi;s_1)$ by using the \emph{dual recursive formula}, starting from an appropriate pair $(x,Q)$, namely $Q$ is the matrix of conditionals corresponding to $\pi$ and $x$ belongs to the sub-differential of $p'\mapsto v_\theta(p'\otimes Q,s)$ at $p$ such that $\pi=p\otimes Q$.  
\end{corollaire}

\subsection{Differential games with incomplete information}

Differential games with incomplete information were introduced by Cardaliaguet \cite{carda07}. As in repeated games with incomplete information, 
before the game starts, a pair of parameters $(k,\ell)$ is drawn according to some commonly known probability distribution $\pi$ on $K\times L$.
  Player $1$ is informed of $k$ and Player 2 of $\ell$.
Then, a differential game is played in which the dynamic and the payoff function depend on both types: each player is thus partially informed about the differential game that is played. The existence and characterisation of the value function was established by Cardaliaguet \cite{carda07} in the independent case, and extended to the general case by Oliu-Barton \cite{OBjeuxdiffs15}. The proof relies on the geometry of the value function ($I$-concavity and $\two$-convexity) and on a sub-dynamic programming principle satisfied by its Fenchel conjugates (i.e. the values of the first and the second dual games). 

Though useful for establishing the existence of the value for differential games with incomplete information, the sub-dynamic programming principles satisfied by the values of the dual games do not yield a construction of optimal strategies for these games, which remains an open problem. Establishing a continuous-time analogue of the dual recursive formula (i.e. Theorem \ref{thmmain}) would be a natural way to solve it.

\section{Acknowledgements}
The authors are indebted to Sylvain Sorin and Bernard De Meyer for their insight and suggestions. 
The authors are also thankful to Mario Bravo for his comments.
The first author gratefully acknowledges support from the  Artificial and Natural
Intelligence Toulouse Institute under Grant ANR-3IA, and funding from the French National Research Agency(ANR), under the Investments for the Future (Investissements d'Avenir) program under grant ANR-17-EURE-0010.
The second author gratefully acknowledges funding from the French National Research Agency (ANR), under grant ANR CIGNE (ANR-15-CE38-0007-01).

\bibliographystyle{amsplain}
\bibliography{biblio_bothsides}

\end{document}